\documentclass{amsart}

\usepackage{amsmath, amsthm, amssymb}
\usepackage{dsfont}
\usepackage{cite}
\usepackage{url}
\usepackage{bm}
\usepackage{multirow}
\usepackage{booktabs}
\usepackage{lineno}

\usepackage{moreverb}
\newtheorem{thm}{Theorem}
\newtheorem{lem}[thm]{Lemma}

\newtheorem{rem}[thm]{Remark}

\newcommand{\HCurlZ}{H_0(\mathbf{curl};\Omega)}
\newcommand{\HDiv}{H(\mathrm{div})}
\newcommand{\HCurl}{H(\mathbf{curl})}

\newcommand{\CURL}{\mathbf{curl}\,}
\newcommand{\bg}{\bm{g}}
\newcommand{\bu}{\bm{u}}
\newcommand{\bv}{\bm{v}}
\newcommand{\bw}{\bm{w}}

\newcommand{\br}{\bm{r}}
\newcommand{\p}{\partial}
\newcommand{\bn}{\bm{n}}
\newcommand{\bz}{\bm{z}}
\newcommand{\bt}{\bm{t}}
\newcommand{\cT}{\mathcal{T}}
\newcommand{\cF}{\mathcal{F}}
\newcommand{\cE}{\mathcal{E}}
\newcommand{\cV}{\mathcal{V}}
\newcommand{\cN}{\mathcal{N}}
\newcommand{\cM}{\mathcal{M}}

\newcommand{\lnorm}[2]{\left\| #1 \right\|_{L^2(#2)}}

\title[MG for 3D $H(\mathbf{curl})$ with DD smoothers]{Multigrid methods for 3$D$ $\HCurl$ problems with nonoverlapping domain decomposition smoothers}

\author{Duk-Soon Oh}
\address{Chungnam National University, 99 Daehak-ro, Yuseong-gu, Daejeon, Republic of Korea.}
\email{duksoon@cnu.ac.kr}
\thanks{This work was supported by research fund of Chungnam National University.}

\begin{document}

\begin{abstract}
We propose V--cycle multigrid methods for vector field problems arising from the lowest order hexahedral N\'{e}d\'{e}lec finite element. Since the conventional scalar smoothing techniques do not work well for the problems, a new type of smoothing method is necessary. We introduce new smoothers based on substructuring with nonoverlapping domain decomposition methods. We provide the convergence analysis and numerical experiments that support our theory.
\end{abstract}

    

\maketitle

\section{Introduction}
In this paper, the following boundary value problem in three dimensions will be considered:
\begin{equation}\label{eq:model problem}
    \begin{aligned}
    L \bu  := \CURL \, (\alpha \, \CURL \, \bu) +  \bu & = \bm{f} \text{ in } \Omega,\\
    \bn \times (\bu \times \bn) & = 0 \text{ on } \partial \Omega.
    \end{aligned}
\end{equation}
Here, $\Omega$ is a bounded convex hexahedral domain in three dimensions whose edges are parallel to the coordinate axes and $\bn$ is the outward unit normal vector of its boundary. We assume that the coefficient $\alpha$ is a strictly positive constant and $\bm{f}$ is in $\left(L^2(\Omega)\right)^3$. 

Our model problem \eqref{eq:model problem} is posed in the Hilbert space $\HCurlZ$, the subspace of $H(\mathbf{curl};\Omega)$ with zero tangential components on the boundary $\partial \Omega$. Here, the space $H(\mathbf{curl};\Omega)$ is defined by
\begin{equation}\label{eq:H(curl)}
    H(\mathbf{curl};\Omega) = \left\lbrace \bu \in \left(L^2(\Omega)\right)^3 :  \CURL \bu \in \left(L^2(\Omega)\right)^3  \right\rbrace.
\end{equation}
Applying integration by parts, the corresponding variational problem for \eqref{eq:model problem} can be obtained as follow:
Find $\bu\in\HCurlZ$ such that
\begin{equation}\label{eq:model variational problem}
 a(\bu, \bv)=(\bm{f}, \bv ) \qquad \forall\,\bv\in \HCurlZ,
\end{equation}
where
\begin{equation}\label{eq:def bilinear form}
\begin{aligned}
 a(\bw,\bv) & := \alpha  \int_{\Omega} \CURL \, \bw \cdot \CURL \,
   \bv \, d\bm{x} + \int_{\Omega} \bw \cdot \bv \, d\bm{x}, \\
(\bm{f}, \bv) & := \int_{\Omega} \bm{f} \cdot \bv \, d\bm{x}.
\end{aligned}
\end{equation}
We will also define the following bilinear forms for any subdomain $D \subset \Omega$ by:
\begin{equation}
    a_D(\bw, \bv) := \alpha  \int_{D} \CURL \, \bw \cdot \CURL \,
    \bv \, d\bm{x} + \int_{D} \bw \cdot \bv \, d\bm{x}
\end{equation}
and
\begin{equation}
    \left(\bw, \bv \right)_D = \int_{D} \bw \cdot \bv \, d\bm{x}.
\end{equation}

The problem \eqref{eq:model problem} is motivated by the eddy-current problem of Maxwell's equation, see \cite{Bossavit:2005, Monk:1991:Maxwell}. Specifically, time-dependent Maxwell's equations satisfy the following system:
\begin{alignat}{3}
& \epsilon \dfrac{\partial }{\partial t} \bm{E} 
 + \sigma \bm{E} && - \CURL \, \bm {H} && = \bm{J} \text{ in } \Omega \times [0, T]\\
& \mu \dfrac{\partial }{\partial t} \bm{H} && + \CURL \, \bm{E} 
&& = 0 \text{ in } \Omega \times [0, T],
\end{alignat}
where $\bm{E}$ is the electric field, $\bm{H}$ is the magnetic field and $\bm{J}$ is the intrinsic current. Eliminating $\bm{H}$ and employing implicit methods yield the following equation in each time step:
\begin{equation}\label{eq:Maxwell timestep}
\dfrac{1}{4} \Delta t^2 \CURL \left(\dfrac{1}{\mu} \CURL \, \bm{E}_n \right) + \left(\epsilon + \dfrac{1}{2} \sigma \Delta t\right) \bm{E}_n = \text{R.H.S.} \text{ in } \Omega.
\end{equation}
The problem \eqref{eq:Maxwell timestep} is equivalent to our model problem \eqref{eq:model problem}. Hence, efficient numerical methods for \eqref{eq:model problem} are essential for solving time-dependent Maxwell's equations. There have been a number of attempts for designing fast solvers related to multigrid methods or domain decomposition methods for the problem \eqref{eq:model problem}. For more details, see \cite{AFW:2000:H(div), Calvo:2020:OSHCurl, Calvo:2015:OSHCurl, Hiptmair:1999:MGHCurl, HT:2000:Vector, HX:2007:HXDecomp, DW:2015:BDDCHCurl, Toselli:2000:OSHCurl, Toselli:2004:FETIDPHCurl, Zampini:2017:BDDCHCurl, KV:2009:HCurl, ZVDK:2018:BDDCHCurl, Zampini:2017:ABDDC, HZ:2003:DDMaxwell, LWXZ:2007:RSCNSS}

Not like the elliptic problems posed in the $H^1$ Hilbert space, multigrid methods for vector field problems posed in $\HDiv$ or $\HCurl$ are challenging. This is because conventional smoothers designed for $H^1$ related problems, e.g., Jacobi, Gauss-Seidel, are not performing well for vector field problems; see \cite{CGP:1993:MLMFEM}. The structures of the null spaces of the differential operators make the hurdle. For the gradient operator, the kernel consists of all constants. However, all gradient fields and all curl fields are the null spaces of the curl and the divergence operators, respectively. Thus, a special treatment for handling the kernels is essential when building multigrid solvers for vector field problems. There have been several approaches in order to overcome the difficulties. In \cite{Hiptmair:1997:MGHDiv, Hiptmair:1999:MGHCurl}, Hiptmair suggested function space splitting methods based on Helmholtz type decompositions. In the algorithms in \cite{Hiptmair:1997:MGHDiv, Hiptmair:1999:MGHCurl}, the smoothing steps have been applied to the decomposed spaces separately, i.e, the range space and the null space. Later, Hiptmair and Xu developed nodal auxiliary space preconditioning techniques based on a new type of regular decomposition in \cite{HX:2007:HXDecomp}. In \cite{AFW:1997:H(div), AFW:2000:H(div), AFW:1998:H(DIV)}, smoothing methods based on geometric substructures have been proposed. Overlapping types of domain decomposition preconditioners have been applied to vector fields successfully. Another class of methods related to nonoverlapping substructure has been considered for $\HDiv$ problems by the author and Brenner in \cite{BO:2018:MGHdiv, BO:2018:MGHdivNE}.

In this paper, we suggest V--cycle multigrid methods for $\HCurl$ vector field problems \eqref{eq:model problem} with smoothers based on nonoverlapping domain decomposition preconditioners. We note that our approaches are $\HCurl$ counterparts of the methods in \cite{BO:2018:MGHdiv, BO:2018:MGHdivNE} and nonoverlapping alternatives of the method in \cite{AFW:2000:H(div)}, which reduce the computational complexity when applying the smoothers.

The rest of this paper is organized as follows. We introduce the edge finite elements for our model problem and the discretized problem in Section~\ref{sec:fem}. The V--cycle multigrid algorithms are presented in Section~\ref{sec:mg}. In Section~\ref{sec:analysis}, We next provide the convergence analysis for the suggested methods. The numerical experiments which support our theory are presented in Section~\ref{sec:numerics}, followed by concluding remarks in Section~\ref{sec:conclusions}.

\section{Finite element discretization}\label{sec:fem}
We introduce a hexahedral triangulation $\cT_h$ of the domain $\Omega$. The edge finite element space, also known as N\'{e}d\'{e}lec finite element space of the lowest order, is defined by
\begin{equation}\label{eq:nedlec space}
N_h := \left \lbrace \bu \,  : \, \bu_{\vert T} \in \mathcal{ND}(T), T \in \cT_h \mbox{ and } \bu \in H(\mathbf{curl};\Omega)  \right \rbrace,
\end{equation}
where
\begin{equation}\label{eq:nedlec form}
    \mathcal{ND}(T) :=
    \begin{bmatrix}
    a_1 + a_2 x_2 + a_3 x_3 + a_4 x_2 x_3\\
    b_1 + b_2 x_3 + b_3 x_1 + b_4 x_3 x_1\\
    c_1 + c_2 x_1 + c_3 x_2 + c_4 x_1 x_2
    \end{bmatrix}
\end{equation}
on each element with twelve constants $\lbrace a_i\rbrace$, $\lbrace b_i\rbrace$ and $\lbrace c_i \rbrace$, $i = 1, 2, 3, 4$; see \cite{Monk:2003:BookMaxwell, Nedlec:1980:FEM} for more details. We note that on each hexahedral element $T$, the tangential components of vector fields of the form \eqref{eq:nedlec form} are constants on the twelve edges of $T$. The twelve coefficients are completely determined by the average tangential components, which is obtained by
\begin{equation}\label{eq:NDdof}
    \lambda_e \left( \bv \right) := \frac{1}{|e|} \int_e \bv \cdot \bt_e  \, ds,
   \end{equation}
on the twelve edges. Here, $e$ is one of the twelve edges of $T$, $\vert e \vert$ is the length of $e$, and $\bt_e$ is the unit tangential vector along the edge $e$. The standard basis function for $N_h$ associated with $e$ is denoted by $\phi_e$. We note that $\lambda_e(\phi_e) = 1$ and $\lambda_{e'}(\phi_e) = 0$ for $e' \neq e$.

Applying the finite element method with $N_h$, the discrete problem for \eqref{eq:model variational problem} is given by the following form: Find $\bu_h \in N_h$ such that
\begin{equation}\label{eq:fem}
    a(\bu_h, \bv) = (\bm{f}, \bv)\qquad \forall\, \bv \in N_h.
\end{equation}

The operator $A_h : N_h \longrightarrow N_h'$ is defined by
\begin{equation}
    \left \langle A_h \bw_h, \bv_h \right \rangle = a(\bw_h, \bv_h) \qquad \forall\, \bv_h, \bw_h \in N_h,
\end{equation}
where $\langle \cdot, \cdot \rangle$ is the canonical bilinear form on $N_h' \times N_h$.
We also define $f_h \in N_h'$ in the following way:
\begin{equation}
    \left \langle f_h, \bv_h \right \rangle = (\bm{f}, \bv_h) \qquad \forall\, \bv_h \in N_h.
\end{equation}
Then, the discrete problem \eqref{eq:fem} can be written as 
\begin{equation}
    A_h \bu_h = f_h.
\end{equation}
\section{Multigrid algorithms}\label{sec:mg}
\subsection{Triangulations and grid transfer operators}\label{subsec:grid}
We introduce $\cT_0$, an initial triangulation of the domain $\Omega$. The triangulations $\cT_1, \cT_2, \cdots$ are obtained from the initial triangulation $\cT_0$ by uniform refinement with the relation $h_k = h_{k-1} / 2$, where $h_k$ is the mesh size of $\cT_k$. The lowest-order N\'{e}d\'{e}lec space associated with $\cT_k$ is denoted by $N_k$. Then, we can rewrite the corresponding $k-$th level discrete problem as
\begin{equation}\label{eq:kthEq}
    A_k \bu_k = f_k.
\end{equation}
In order to design V--cycle multigrid methods for solving \eqref{eq:kthEq}, two essential ingredients, i.e., intergrid transfer operators and smoothers, have to be defined properly. We first focus on the grid transfer operators. Due to the fact that the finite element spaces are nested, we can use the natural injection to define the coarse-to-fine operator $I_{k-1}^k : N_{k-1} \longrightarrow N_k$. The associated fine-to-coarse operator $I_{k}^{k-1} : N_k' \longrightarrow N_{k-1}'$ can be defined by 
\begin{equation}\label{eq:F2C}
    \langle I_k^{k-1}\ell,\bv\rangle=\langle\ell,I_{k-1}^k\bv\rangle
    \qquad\forall\,\ell\in N_k',\,\bv\in N_{k-1}.
\end{equation}

\subsection{Smoothers}\label{subsec:smoothers}
We now concentrate on the other ingredient, smoothers. Nonoverlapping type domain decomposition methods will be used to construct the smoothers. In order to keep consistency with the notations for the standard two-level domain decomposition methods, we will denote $\cT_{k-1}$ by $\cT_H$ and $\cT_k$ by $\cT_h$. It means that
all the coarse level and fine level settings are associated with $\cT_{k-1} ( = \cT_H)$ and $\cT_{k}(=\cT_h)$, respectively. We also define geometric substructures. We will use $\cF_{H}$, $\cE_{H}$, and $\cV_{H}$ to denote the sets of interior faces, edges, and vertices of $\cT_{H}$, respectively. We also define $\cE_h^D$ for any subdomain $D \subset \Omega$ by the set of interior edges associated with $\cT_h$ that are parts of $D$. Similarly, we define $\cV_h^D$ by the set of interior vertices related to $\cT_h$ that are contained in $D$.

We first introduce the interior space. For each element $T \in \cT_{H}$, we define the following subspace $N_{h}^T$ of $N_h$:
\begin{equation}\label{eq:SubdomainSpace}
    N_h^T=\{\bv\in N_h:\,\bv=\bm{0}\;\text{on}\; \Omega\setminus T\}.
\end{equation} 
We next denote by $J_T$ the natural injection from $N_h^T$ into $N_h$ and we define the operator $A_T:N_h^T\longrightarrow (N_h^T)'$ by
\begin{equation}\label{eq:ATDef}
    \langle A_T \bw, \bv \rangle = a(\bw, \bv) \qquad \forall\,\bv,\bw\in N_h^T.
\end{equation}
In the rest of this subsection, we will introduce two types of smoothing techniques, edge-based and vertex-based smoothers.

\subsubsection{Edge-based smoothers}\label{subsubsec:edge-based}
We first consider an edge-based smoother. For a given edge $E \in \cE_H$, we can find four elements, $\left \lbrace T_E^i\right \rbrace_{i = 1, 2, 3, 4}$ in $\cT_H$, and four faces, $\left \lbrace F_E^i\right \rbrace_{i = 1, 2, 3, 4}$ in $\cF_H$, that are sharing the edge $E$. We define the edge space $N_h^E$ of $N_h$ by 
\begin{equation}\label{eq:edge space}
	\begin{aligned}
		N_h^E = & \left\{ \bv \in N_h : \bv \cdot \bt_e = 0 \mbox{ for } e \in \cE_h^{\Omega} \left\backslash  \left(\left(\cup_{i=1}^4 \cE_h^{T_E^i}\right) \bigcup \left(\cup_{j=1}^4 \cE_h^{F_E^j}\right) \bigcup \cE_h^E \right) \right., \right.\\ 
		& \left. \hspace{10pt} \mbox{ and } a(\bv, \bw)  = 0 \quad \forall \, \bw \in 
		\left( N_h^{T_E^1} + N_h^{T_E^2} + N_h^{T_E^3} + N_h^{T_E^4} \right) \right\}.
	\end{aligned}
\end{equation}

We remark that due to \eqref{eq:edge space}, if $\bv\in N_h^E$ and $\bw$ have the same tangential components as $\bv$ on the edges associated with $\partial T_E^i, i = 1, 2, 3, 4$, we have the following property:
\begin{equation}\label{eq:MinimumEnergyE}
    a_{T_E^i}\left(\bv, \bv\right) \le a_{T_E^i}\left(\bw, \bw\right), \qquad i = 1, 2, 3, 4.
\end{equation}

The operator $J_E : N_h^E \longrightarrow N_h$ is defined as the natural injection. We next define the operator $A_E : N_h^E \longrightarrow \left( N_h^E\right)'$ by 
\begin{equation}\label{eq:AEDef}
	\langle A_E\bw,\bv\rangle= a(\bw,\bv)\qquad\forall\,\bv,\bw\in N_h^E.
\end{equation}
The edge-based smoothing operator $M_{E, h}^{-1}$ is constructed as follow:
\begin{equation}\label{eq:Preconditioner E}
	M_{E,h}^{-1}=\eta_E\left(\sum_{T\in\cT_{H}}J_T A_T^{-1} J_T^t+\sum_{E\in\cE_{H}}J_E A_E^{-1} J_E^t\right),
\end{equation}
where $\eta_E$ is a damping factor and $J_T^t : N_h' \longrightarrow \left(N_h^T\right)'$ and $J_E^t : N_h' \longrightarrow \left(N_h^E\right)'$ are the transposes of $J_T$ and $J_E$, respectively. We can choose the damping factor $\eta_E$ such that the spectral radius of $M_{E, h}^{-1} A_h \le 1$. We note that the condition is satisfied if $\eta_E \le 1/12$, which are assumed to be the case from now on. 

\subsubsection{Vertex-based smoothers}\label{subsubsec:vertex-based}
We now consider a vertex-based method. In order to define the vertex space $N_h^V$, we need geometric substructures associated with the given coarse vertex $V \in \cV_H$. For each $V \in \cV_H$, there are eight elements, $\left \lbrace T_V^i\right \rbrace_{i = 1, \cdots, 8}$ in $\cT_H$, twelve faces, $\left \lbrace F_V^i\right \rbrace_{i = 1, \cdots, 12}$ in $\cF_H$, and six edges, $\left \lbrace E_V^i\right \rbrace_{i = 1, \cdots, 6}$ in $\cE_H$, that have the vertex $V$ in common. The vertex space $N_h^V$ is defined by 
\begin{equation}\label{eq:vertex space}
	\begin{aligned}
		N_h^V = & \left\{ \bv \in N_h : \bv \cdot \bt_e = 0 \mbox{ for } e \in \cE_h^\Omega \left.\backslash \left(\left(\cup_{i=1}^8 \cE_h^{T_V^i}\right) \bigcup \left(\cup_{j=1}^{12} \cE_h^{F_V^j}\right) \bigcup \left(\cup_{l=1}^6 \cE_h^{E_V^l}\right) \right),\right. \right.\\ 
		& \left. \hspace{10pt} \mbox{ and } a(\bv, \bw)  = 0 \quad \forall \, \bw \in 
		\left( \sum_{i=1}^8 N_h^{T_V^i}\right) \right\}.
	\end{aligned}
\end{equation}

Note that \eqref{eq:vertex space} implies the following minimum energy property:
\begin{equation}\label{eq:MinimumEnergyV}
    a_{T_V^i}\left(\bv, \bv\right) \le a_{T_V^i}\left(\bw, \bw\right), \qquad i = 1, \cdots, 8
\end{equation}
for $\bv \in N_h^V$ and $\bw \in N_h$ with the same degrees of freedom as $\bv$ on $\partial T_V^i, i = 1, \cdots, 8$.

The vertex-based preconditioner is given by 
\begin{equation}\label{eq:Preconditioner P}
	M_{V,h}^{-1}=\eta_V(\sum_{T\in\cT_{H}}J_T A_T^{-1} J_T^t+\sum_{V\in\cV_{H}}J_V A_V^{-1} J_V^t).
\end{equation}
Here, $\eta_V$ is a damping factor and $J_V$, $J_V^t$, and $A_V$ are defined in a similar way to those in the edge-based method. The operator $J_V : N_h^V \longrightarrow N_h$ is the natural injection and $J_V^t : N_h' \longrightarrow \left(N_h^V\right)'$ is the transpose of $J_V$. We define $A_V : N_h^V \rightarrow \left(N_h^V\right)'$ as follow:
\begin{equation}\label{eq:AVDef}
	\langle A_V\bw,\bv\rangle= a(\bw,\bv)\qquad\forall\,\bv,\bw\in N_h^V.
\end{equation}
We note that if $\eta_V \le 1/8$, the spectral radius of $M_{V, h}^{-1} A_h \le 1$ and we will use the condition for the rest of this paper.

\subsection{V--cycle multigrid algorithm}
Combining all together, we now construct the symmetric V--cycle multigrid algorithm. Let $MG(k, g, \bz_0, m)$ be the output of the $k-$th level symmetric multigrid algorithm for solving $A_k \bz = g$ with initial guess $\bz_0 \in N_k$ and $m$ smoothing steps. The algorithm is defined in Figure~\ref{fig:vmg}.
\begin{figure}[htb]
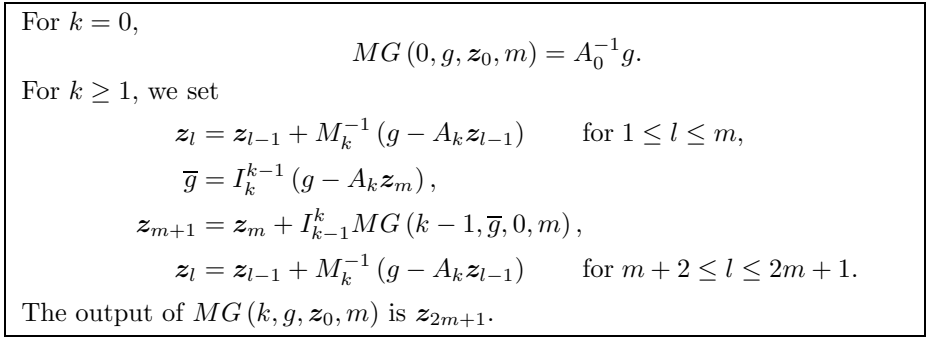

    \centering
        \framebox[0.95\width][l]{
        \begin{minipage}{\textwidth}
            For $k=0$,
            \begin{equation*}
                MG\left(0, g, \bz_0, m\right) = A_0^{-1}g.\nonumber
            \end{equation*}
            \par\noindent
            For $k\ge1$, we set
            \begin{align*}
            \bz_l &= \bz_{l-1} + M_k^{-1} \left( g - A_k \bz_{l-1} \right) \qquad \text{for}\; 1 \le l \le m,\nonumber\\
            \overline{g} &= I_{k}^{k-1} \left(g - A_k \bz_{m} \right),\nonumber\\
            \bz_{m+1} &= \bz_{m} + I_{k-1}^k MG\left(k-1, \overline{g}, 0, m\right), \nonumber\\
            \bz_l &= \bz_{l-1} + M_k^{-1} \left( g - A_k \bz_{l-1} \right)\qquad\text{for}\; m+2 \le l \le  2m+1.\nonumber
            \end{align*}
            The output of $MG\left(k, g, \bz_0, m\right)$ is $\bz_{2m + 1}$.
        \end{minipage}
        }
        \caption{V--cycle Multigrid Method}\label{fig:vmg}
\end{figure}

The smoothing operator $M_k^{-1}$ will be either $M_{E,k}^{-1}$ or $M_{V, k}^{-1}$. We note that given $\ell \in N_k'$, the cost of computing $M_k^{-1} \ell$ is $O(n_k)$ for both edge-based and vertex based smoothers, where $n_k$ is the number of degrees of freedom of $N_k$. Hence, the overall computational complexity for $MG(k, g, \bz_0, m)$ is also $O(n_k)$.
\section{Convergence analysis}\label{sec:analysis}
We first define operators that are useful for our analysis. The projection operator $P_H$ is defined by the Ritz projection from the fine level space $N_h$ to the coarse level space $N_H$ with respect to the bilinear form $a(\cdot, \cdot)$ and the identity operator on $N_h$ is denoted by $I$.

We will also need the Lagrange finite element space of order one, $W_h$ for our analysis. The degree of freedoms are chosen as the function evaluations at vertex points and are denoted by $\nu_v(p) := p(v)$. The standard basis function associated with the vertex $v$ is denoted by $\psi_v$, i.e., $\nu_v(\psi_v) = 1$ and $\nu_{v'}(\psi_v) = 0$ for $v' \neq v$.

\subsection{Stability estimates}
The next lemma, which is useful for the stability in the edge space, can be obtained by a direct calculation.
\begin{lem}\label{lem:coarse}
For a given coarse edge $E \in \cE_H$, which is parallel to the $x_1$ axis, there are four elements $T_E^i \in \cT_H, i = 1, 2, 3, 4$. Let $v$ be the midpoint of $E$. Then, there are six fine edges $e_i \in \cE_h, i = 1, \cdots, 6$, that share $v$. Let $e_{2i -1}$ and $e_{2i}$ be parallel to the $x_i$ axis for $i = 1, 2, 3$ and let us fix the directions for the tangential vectors, $\bt_{x_i}, i = 1, 2, 3$, for all corresponding fine edges. Without loss of generality, let $v$ be the endpoint of $e_1$, $e_3$, and $e_5$ with respect to the given tangential directions. We construct $\bu \in N_h\left(\cup_{i=1}^4 T_E^i\right)$ by the properties that 
\begin{itemize}
    \item $\bu \cdot \bt_{x_i} = -1$ on $e_{2i -1}$ and $\bu \cdot \bt_{x_i} = 1$ on $e_{2i}$ for $i = 1, 2, 3$.
    \item The tangential component of $\bu$ vanishes on the other edges in $\bigcup_{i=1}^4 \cE_h^{\p T_E^i}$.
    \item $\bu$ is orthogonal to $N_h^{T_E^i}$ with respect to the innerproduct $\left(\cdot, \cdot\right)_{T_E^i}$ for $i = 1, 2, 3, 4$.
\end{itemize}
Then, $\CURL\bu$ does not vanish.
\end{lem}

In \cite{AFW:2000:H(div)}, Arnold, Falk, and Winther suggested the following discrete orthogonal Helmholtz decomposition that plays an essential role in the analysis. 
\begin{lem}\label{lem:Helmholtz Decomposition}[Discrete Helmholtz decomposition] For any $\bw \in \left(I - P_H\right)N_h$, there exist $\br \in N_h$ and $q \in W_h$ such that
\begin{equation}\label{eq:HD}
\bw = \br + \nabla \, q,
\end{equation}
and
\begin{align}
    \lnorm{\br}{\Omega}^2+ \lnorm{\nabla \, q}{\Omega}^2
    &= \lnorm{\bw}{\Omega}^2, \label{eq:HD1}\\
    \alpha \lnorm{\br}{\Omega}^2 & \le C H^2 a \left( \bm{w}, \bm{w} \right),   \label{eq:HD2}\\
    \lnorm{q}{\Omega}^2 &\le C H^2 \lnorm{\bw}{\Omega}^2,\label{eq:HD3}
\end{align}
where the positive constant $C$ does not depend on the mesh size $h$.
\end{lem}

\begin{rem}\label{rem:coarse edge decomposition}
For given $\bw \in N_h$ and $E \in \cE_H$, we find that 
\begin{equation}\label{eq:coarse edge decomposition}
\bw \cdot \bt_e = \lambda_E(\bw) \phi_E \cdot \bt_e + \nabla p_E \cdot \bt_e = \lambda_E(\bw) + \nabla p_E \cdot \bt_e, \quad \forall e \in \cE_h^E.
\end{equation}
Here, $p_E \in W_h$ is a constant multiple of the standard basis functions of $W_h$ associated with the interior node of $E$. For more details, see (6.1) of \cite{Toselli:2004:FETIDPHCurl} and (2.4) of \cite{DW:2015:BDDCHCurl}.
\end{rem}

The edge-based smoother has the following stable decomposition result:
\begin{lem}\label{lem:StabilityE}
For any $\bw \in \left( I - P_H \right) N_h$, there exist a constant $C_{E,\dag}$ that does not depend on $\alpha$, $h$ and the number of elements in $\cT_H$ and
a decomposition
\begin{equation*}
    \bw = \sum_{T \in \cT_H} \bw_T + \sum_{E \in \cE_H} \bw_E,
\end{equation*}
such that
\begin{equation}\label{eq:StabilityEstimateE}
    \sum_{T \in \cT_H} a \left( \bw_T , \bw_T \right) + \sum_{E \in \cE_H} a \left( \bw_E , \bw_E \right) \le C_{E,\dag} a\left( \bw, \bw \right).
\end{equation}
\end{lem}
\begin{proof}
For given $\bw \in (I - P_h)N_h$, we consider the decomposition \eqref{eq:HD} in Lemma~\ref{lem:Helmholtz Decomposition}, i.e., $\bw = \br + \nabla \, q$.

For each coarse edge $E \in \cE_H$, we have four coarse faces $F_E^i \in \cF_H, i = 1, 2, 3, 4$ and four elements $T_E^i \in \cT_H, i = 1, 2, 3, 4$, that are sharing $E$. We denote by $\cN_{F_E^i}$ the number of edges in $\cE_H$ that are parts of $\p F_E^i$. We now construct $\br_{E, F} \in N_h^E$ and $\br_{E, E} \in N_h^E$ in the following way:
\begin{equation}
    \br_{E, F} \cdot \bt_e =
        \dfrac{1}{\cN_{F_E^i}} \br \cdot \bt_e  \mbox{ for } e \in \cE_h^{F_E^i} \qquad \mbox{ for } i = 1, 2, 3, 4,
\end{equation}
\begin{equation}
    \br_{E, E} \cdot \bt_e = 
        \br \cdot \bt_e  \qquad \mbox{ for } e \in \cE_h^E,
\end{equation}
and \eqref{eq:edge space}. Then, $\br$ and $\sum_{E \subset \cE_H} \left(\br_{E, F} + \br_{E, E}  \right)$ have identical degrees of freedom on the edges contained in the boundaries of elements in $\cT_H$. Thus, we can find $\br_T \in N_h^T$ such that
\begin{equation}
    \br = \sum_{T \in \cT_H} \br_T + \sum_{E \in \cE_H} \left(\br_{E, F} + \br_{E, E}\right).
\end{equation}
We first consider the vector fields associated with the interior spaces $N_h^T$. 
We note that the interior spaces are orthogonal to all the edge spaces $N_h^E$ with respect to the bilinear form $a(\cdot, \cdot)$. Also, the interior spaces are mutually orthogonal. Thus, we have the following estimate putting together with \eqref{eq:HD1}, \eqref{eq:HD2}, and a standard inverse inequality:
\begin{align}\label{eq:rTEstE}
    \sum_{T \in \cT_H}a\left(\br_T,\br_T\right) & = a\left(\sum_{T\in\cT_H}\br_T,\sum_{T\in\cT_H}\br_T\right)\notag\\
          &\leq a(\br,\br)\notag\\
          &= \sum_{T\in\cT_H}\left(\alpha\lnorm{\CURL\br}{T}^2+\lnorm{\br}{T}^2\right)\\
          &\leq \sum_{T\in\cT_H}\left(C\dfrac{\alpha}{h^2}\lnorm{\br}{T}^2+\lnorm{\br}{T}^2\right)
          \leq Ca(\bw,\bw).\notag
\end{align}

We next consider the vector fields associated with edges. For any $E \in \cE_H$, we construct $\widetilde{\br}_{E, F}$ in the following way:
\begin{equation}
\widetilde{\br}_{E, F} = \sum_{e \in  \cM} \lambda_e(\br_{E, F}) \phi_e,
\end{equation}
where $\cM = \cup_{i=1}^4 \cE_h^{F_E^i}$. From \eqref{eq:MinimumEnergyE} and a scaling argument, we obtain
\begin{equation}\label{eq:rEEst1}
a(\br_{E, F}, \br_{E, F}) \le a(\widetilde{\br}_{E, F}, \widetilde{\br}_{E, F})
\end{equation}
and
\begin{equation}\label{eq:rEEst2}
\lnorm{\widetilde{\br}_{E, F}}{T_E^i} \le C \lnorm{\br}{T_E^i}, \qquad i = 1, 2, 3, 4.
\end{equation}
Using \eqref{eq:HD1}, \eqref{eq:HD2}, \eqref{eq:rEEst1}, \eqref{eq:rEEst2}, and an inverse inequality, we obtain
\begin{equation}\label{eq:rEEstE}
\begin{aligned}
    \sum_{E \in \cE_H} a(\br_{E, F}, \br_{E, F}) & \le \sum_{E \in \cE_H} a(\widetilde{\br}_{E, F}, \widetilde{\br}_{E, F}) \\
    & = \sum_{E \in \cE_H} \sum_{i=1}^4 \left[\alpha \lnorm{\CURL \, \widetilde{\br}_{E, F}}{T_E^i}^2 + \lnorm{\widetilde{\br}_{E, F}}{T_E^i}^2 \right] \\
    & \le C \sum_{E \in \cE_H} \sum_{i=1}^4 \left[\dfrac{\alpha}{h^2} \lnorm{\br}{T_E^i}^2 + \lnorm{\br}{T_E^i}^2\right] \le C a(\bw, \bw).
\end{aligned}
\end{equation}
We therefore have by \eqref{eq:rTEstE} and \eqref{eq:rEEstE} 
\begin{equation}\label{eq:bEstAll}
    \sum_{T \in \cT_H} a(\br_T, \br_T) + \sum_{E \in \cE_H} a(\br_{E, F}, \br_{E, F}) \le C a(\bw, \bw).
\end{equation}

Let $\bg = \nabla\, q$. We construct $\bg_{E, F}$ and $\bg_{E, E}$ in exactly the same way with $\br_{E, F}$ and $\br_{E, E}$. Now that $\bg$ and $\sum_{E \in \cE_H} (\bg_{E, F} + \bg_{E, E})$ have the same degrees of freedom on the edges of $N_H$, we have
\begin{equation}
    \bg = \sum_{T \in \cT_H} \bg_T + \sum_{E \in \cE_H} \left(\bg_{E, F} + \bg_{E, E} \right)
\end{equation}
for unique vector fields $\bg_T \in N_h^T$. Then, the orthogonal properties and \eqref{eq:HD1} imply the estimate
\begin{align}\label{eq:gTEstE}
    \sum_{T \in \cT_H}a\left(\bg_T,\bg_T\right) & = a\left(\sum_{T\in\cT_H}\bg_T,\sum_{T\in\cT_H}\bg_T\right)\notag\\
    & \le a(\bg, \bg) = \lnorm{\nabla q}{\Omega}^2 \le \lnorm{\bw}{\Omega}^2 \le a(\bw, \bw).
\end{align}

For a scalar function $z$, we define $z_v$ for $v \in \cV_h$ by 
\begin{equation}
    z_v := \nu_v(z)\psi_v.
\end{equation}
For a given $E \in \cE_H$, let $v_E \in \cV_h$ be the midpoint of $E$. Similarly, we denote by $v_F \in \cV_h$ the midpoint of $F\in \cF_H$.

For each $E \in \cE_H$, we now construct $\widetilde{\bg}_{E, F}^{(1)} \in N_h$ and $\widetilde{\bg}_{E, F}^{(2)} \in N_h^E$. The vector field $\widetilde{\bg}_{E, F}^{(1)}$ is defined by
\begin{equation}
    \widetilde{\bg}_{E, F}^{(1)} = \nabla \, \left( \sum_{i=1}^4 \dfrac{1}{\cN_{F_E^i}} q_{v_{F_E^i}}\right).
\end{equation}

For $e \in \cE_h^{F_E^i}$, let $E_e^i \in \cE_H^{\p F_E^i}$ be the coarse edge that shares one vertex point with $e$.
Then, $\widetilde{\bg}_{E, F}^{(2), e} \in N_h^E$ is defined by 
\begin{equation}
\widetilde{\bg}_{E, F}^{(2), e} \cdot \bt_e = \nabla \, \left(\dfrac{1}{\cN_{F_E^i}} q_{v_{E_e^i}}\right)\cdot \bt_e\qquad \mbox{ for } e,
\end{equation}
\begin{equation}
    \widetilde{\bg}_{E, F}^{(2), e} \cdot \bt_{e'} = 0 \mbox{ for } e' \in \cE_h^{F_E^i} \backslash \lbrace e \rbrace \qquad \mbox{ for } i = 1, 2, 3, 4,
\end{equation}
and \eqref{eq:edge space}. We construct $\widetilde{\bg}_{E, F}^{(2)}$ as follow:
\begin{equation}\label{eq:gEFEst3}
    \widetilde{\bg}_{E, F}^{(2)} = \sum_{i=1}^4 \sum_{e \in \cE_h^{F_E^i}} \widetilde{\bg}_{E, F}^{(2), e}.
\end{equation}
We note that $\widetilde{\bg}_{E, F}^{(1)} + \widetilde{\bg}_{E, F}^{(2)}$ and $\bg_{E, F}$ have the same degrees of freedom on the edges in $\left(\cup_{j=1}^4 \cE_h^{F_E^j}\right) \bigcup \cE_h^E$.

We first estimate $\widetilde{\bg}_{E, F}^{(1)}$. By a standard inverse inequality and a scaling argument, we obtain
\begin{equation}\label{eq:gEFEst1}
    \lnorm{\widetilde{\bg}_{E, F}^{(1)}}{T_E^i}^2 \le \dfrac{C}{h^2}\lnorm{q}{T_E^i}^2, \qquad i = 1, 2, 3, 4.
\end{equation}
Hence, from \eqref{eq:HD3}, and \eqref{eq:gEFEst1} we have
\begin{equation}\label{eq:gEFEstAll1}
    \sum_{E \in \cE_H} a(\widetilde{\bg}_{E, F}^{(1)}, \widetilde{\bg}_{E, F}^{(1)}) \le \dfrac{C}{h^2}\lnorm{q}{\Omega}^2 \le C \lnorm{\bw}{\Omega}^2 \le C a(\bw, \bw).
\end{equation}

We next consider $\widetilde{\bg}_{E, F}^{(2)}$. For each $v_{E_e^i}$, there exist six fine edges $\lbrace e_i \rbrace, i = 1, \cdots , 6$ in $\cE_h$ that have $v_{E_e^i}$ in common. We define $\widehat{\bg}_{E,F}^{(2), e} \in N_h^E$ as follow:
\begin{equation}
    \widehat{\bg}_{E,F}^{(2), e} \cdot \bt_{e'} = \nabla \, \left(\dfrac{1}{\cN_{F_E^i}} q_{v_{E_e^i}}\right) \bt_{e'}, \qquad \mbox{ for } e' = e_i, i = 1, \cdots, 6
\end{equation}
and \eqref{eq:edge space}. We compare $\widetilde{\bg}_{E,F}^{(2), e}$ and $\widehat{\bg}_{E,F}^{(2), e}$. Let $\lbrace T_{E_e^i}^j\rbrace, j = 1, 2, 3, 4$ be four elements in $\cT_H$, that are sharing $E_e^i$. Because $\lnorm{\widehat{\bg}_{E,F}^{(2), e}}{T_{E_e^i}^j} = 0$ if and only if $\widetilde{\bg}_{E,F}^{(2), e} = 0$, we obtain 
\begin{equation}\label{eq:gEFEst4}
    \lnorm{\widetilde{\bg}_{E,F}^{(2), e}}{T_{E_e^i}^j} \le C \lnorm{\widehat{\bg}_{E,F}^{(2), e}}{T_{E_e^i}^j}, \qquad j = 1, 2, 3, 4.
\end{equation}
Furthermore, it follows from Lemma~\ref{lem:coarse} that $\CURL \, \widehat{\bg}_{E,F}^{(2), e} = 0$ if and only if $\widetilde{\bg}_{E,F}^{(2), e} = 0$. Thus, we have, by a scaling argument again, 
\begin{equation}\label{eq:gEFEst5}
    \lnorm{\CURL \, \widetilde{\bg}_{E,F}^{(2), e}}{T_{E_e^i}^j} \le C \lnorm{\CURL \, \widehat{\bg}_{E,F}^{(2), e}}{T_{E_e^i}^j}, \qquad j = 1, 2, 3, 4.
\end{equation}
Additionally, the construction of $\widehat{\bg}_{E,F}^{(2), e}$, \eqref{eq:MinimumEnergyE}, a scaling argument, and an inverse estimate give the estimate
\begin{equation}\label{eq:gEFEst6}
    a_{T_{E_e^i}^j}(\widehat{\bg}_{E,F}^{(2), e}, \widehat{\bg}_{E,F}^{(2), e}) \le \dfrac{C}{h^2}\lnorm{ q}{T_{E_e^i}^j}^2, \qquad j = 1, 2, 3, 4.
\end{equation}
By summing over all $E \in \cE_H$, $i$, and $e\in\cE_h^{F_E^i}$ and by \eqref{eq:gEFEst3}, \eqref{eq:gEFEst4}, \eqref{eq:gEFEst5}, \eqref{eq:gEFEst6}, and Cauchy-Schwarz inequality, we have
\begin{equation}\label{eq:gEFEstAll2}
    \begin{aligned}
        \sum_{E \in \cE_H} a\left(\widetilde{\bg}_{E, F}^{(2)}, \widetilde{\bg}_{E, F}^{(2)} \right) 
        & \le C \sum_{E \in \cE_H} \sum_{i=1}^4 \sum_{e \in \cE_h^{F_E^i}} a\left(\widetilde{\bg}_{E,F}^{(2), e}, \widetilde{\bg}_{E,F}^{(2), e}\right)  \\
        & \le \dfrac{C}{h^2}\lnorm{q}{\Omega}^2 \le C\lnorm{\bw}{\Omega}^2 \le C a(\bw, \bw).
    \end{aligned}
\end{equation}
Putting all together with \eqref{eq:MinimumEnergyE}, \eqref{eq:gEFEstAll1}, \eqref{eq:gEFEstAll2}, and Cauchy-Schwarz inequality, we obtain
\begin{equation}\label{eq:gEFEstAll3}
    \sum_{E \in \cE_H} a(\bg_{E, F}, \bg_{E, F}) \le C a(\bw, \bw).
\end{equation}

We finally consider the missing piece $\bw_{E, E} = \br_{E, E} + \bg_{E, E}$. Based on the decomposition \eqref{eq:coarse edge decomposition} in Remark~\ref{rem:coarse edge decomposition} and the fact that $\bw \in (I - P_H)N_h$, $\lambda_E(\bw_{E,E}) = 0$ and we then have
\begin{equation}
    \bw \cdot \bt_e = \bw_{E, E} \cdot \bt_e = \nabla \, p_E \cdot \bt_e,
\end{equation}
on $E$ for some $p_E \in W_h$. We note that there is only one degree of freedom for $p_E$.

For $v_E$, there are six edges $\lbrace e_i \rbrace, i = 1, \cdots, 6$ in $\cE_h$ that share $v_E$. We construct $\widehat{\bw}_{E, E}\in N_h^E$ in the following way:
\begin{equation}
    \widehat{\bw}_{E, E} \cdot \bt_e =
        \nabla \, p_E \cdot \bt_e \qquad \mbox{ for } e = e_i, i = 1, \cdots, 6\\
\end{equation}
and \eqref{eq:edge space}. Since $\lnorm{\widehat{\bw}_{E, E}}{T_E^i} = 0$ if and only if $\bw_{E, E} = 0$, by a scaling argument, we have
\begin{equation}\label{eq:wEEEst1}
    \lnorm{\bw_{E, E}}{T_E^i}^2 \le C \lnorm{\widehat{\bw}_{E, E}}{T_E^i}^2, \qquad i = 1, 2, 3, 4.
\end{equation}
Moreover, from Lemma~\ref{lem:coarse}, $\CURL \, \widehat{\bw}_{E, E} = 0$ if and only if $\bw_{E, E} = 0$. We therefore have
\begin{equation}\label{eq:wEEEst2}
    \lnorm{\CURL \, \bw_{E, E}}{T_E^i}^2 \le C \lnorm{\CURL \, \widehat{\bw}_{E, E}}{T_E^i}^2, \qquad i = 1, 2, 3, 4.
\end{equation}
In addition, by the construction of $\widehat{\bw}_{E, E}$, \eqref{eq:MinimumEnergyE}, we have
\begin{equation}\label{eq:wEEEst3}
    a_{T_E^i}(\widehat{\bw}_{E, E}, \widehat{\bw}_{E, E}) \le \lnorm{\nabla \, p_E}{T_E^i}^2, \qquad i = 1, 2, 3, 4. 
\end{equation}
Thus, we obtain by \eqref{eq:wEEEst1}, \eqref{eq:wEEEst2}, \eqref{eq:wEEEst3}, and a scaling argument,
\begin{equation}\label{eq:wEEEstAll}
    \begin{aligned}
        \sum_{E\in\cE_H} a(\bw_{E, E}, \bw_{E, E}) & \le C \sum_{E \in \cE_H} a(\widehat{\bw}_{E, E}, \widehat{\bw}_{E, E}) \\
        & \le C \lnorm{\nabla \, p_E}{\Omega}^2 \le C \lnorm{\bw}{\Omega}^2 \le C a(\bw, \bw).
    \end{aligned}
\end{equation}
With $\bw_T = \br_T + \bg_T$ and $\bw_E = \br_{E, F} + \bg_{E, F} + \bw_{E, E}$, we have the estimate \eqref{eq:StabilityEstimateE} by \eqref{eq:bEstAll}, \eqref{eq:gTEstE}, \eqref{eq:gEFEstAll3},  and \eqref{eq:wEEEstAll}.
\end{proof}

The following lemma shows a stability estimate for the vertex-based method:
\begin{lem}\label{lem:StabilityV}
For any $\bw \in \left( I - P_H \right) N_h$, we can find a decomposition
\begin{equation*}
    \bw = \sum_{T \in \cT_H} \bw_T + \sum_{V \in \cV_H} \bw_V
\end{equation*}
and a constant $C_{V,\dag}$ that does not depend on $\alpha$, $h$  and the number of elements in $\cT_H$, such that
\begin{equation}\label{eq:StabilityEstimateV}
    \sum_{T \in \cT_H} a \left( \bw_T , \bw_T \right) + \sum_{V \in \cV_H} a \left( \bw_V , \bw_V \right) \le C_{V,\dag} a\left( \bw, \bw \right).
\end{equation}
\end{lem}
\begin{proof}
We will consider $\br$ and $\nabla \, q$ in \eqref{eq:HD} separately as in the approach for Lemma~\ref{lem:StabilityE}. 
For each $V \in \cV_H$, we consider the geometric structures $\left \lbrace T_V^i\right \rbrace_{i = 1, \cdots, 8}$ in $\cT_H$, twelve faces, $\left \lbrace F_V^i\right \rbrace_{i = 1, \cdots, 12}$ in $\cF_H$, and six edges, $\left \lbrace E_V^i\right \rbrace_{i = 1, \cdots, 6}$ in $\cE_H$, considered in Section~\ref{subsubsec:vertex-based}. The numbers $\cN_{F_V^i}$ and $\cN_{E_V^j}$ are denoted by the numbers of vertices in $\cV_H$ that are parts of $\p F_V^i$ and $\p E_V^j$, respectively. We now construct $\br_V \in N_h^V$ in the following way:
\begin{equation}
    \br_V \cdot \bt_e
    = \begin{cases}
        \dfrac{1}{\cN_{F_V^i}} \br \cdot \bt_e & \mbox{ for } e \in \cE_h^{F_V^i}, i = 1, \cdots, 12,\\
        \dfrac{1}{\cN_{E_V^j}} \br \cdot \bt_e & \mbox{ for } e \in \cE_h^{E_V^j}, i = 1, \cdots, 6,
    \end{cases}
\end{equation}
and \eqref{eq:vertex space}. We note that $\br - \sum_{V \in \cV_H} \br_V$ belongs to $\sum_{T \in \cT_H} N_h^T$ since $\br$ and $\sum_{V \in \cV_H} \br_V$ have the same degrees of freedom on the edges contained in $\p T, T \in \cT_H$. Hence, we have the following decomposition:
\begin{equation}
\br = \sum_{T \in \cT_H} \br_T + \sum_{V \in \cV_H} \br_V.
\end{equation}
Using the same arguments in \eqref{eq:rTEstE}, we have
\begin{equation}\label{eq:rTEstV}
\sum_{T \in \cT_H} a(\br_T, \br_T) \le C a(\bw, \bw).
\end{equation}
Let $\widetilde{\br}_V$ be defined by
\begin{equation}
    \widetilde{\br}_V := \sum_{i = 1}^{12}\sum_{e \in \cE_h^{F_V^i}} \lambda_e(\br_V) \phi_e + \sum_{j=1}^{6} \sum_{e \in \cE_h^{E_V^j}} \lambda_e(\br_V) \phi_e.
\end{equation}
We then have
\begin{equation}\label{eq:brRelationV}
    a(\br_V, \br_V) \le a(\widetilde{\br}_V, \widetilde{\br}_V)
\end{equation}
and
\begin{equation}\label{eq:tbrL2Est V}
    \lnorm{\widetilde{\br}_V}{T_V^i} \le C \lnorm{\br}{T_V^i}, \qquad i = 1, \cdots, 8
\end{equation}
by \eqref{eq:MinimumEnergyV} and a standard scaling argument.
Combining \eqref{eq:HD1}, \eqref{eq:HD2}, \eqref{eq:brRelationV}, \eqref{eq:tbrL2Est V}, and an inverse estimate, we obtain 
\begin{equation}\label{eq:rVEstAll}
    \begin{aligned}
    \sum_{V \in \cV_H} a\left( \br_V , \br_V \right) & \le \sum_{V \in \cV_H} a(\widetilde{\br}_V, \widetilde{\br}_V) \\
    & = \sum_{V \in \cV_H} \sum_{i=1}^8 \left[ \alpha \lnorm{\CURL \widetilde{\br}_V}{T_V^i}^2 + \lnorm{\widetilde{\br}_V}{T_V^i}^2\right]\\
    & \le \sum_{V \in \cV_H} \sum_{i=1}^8 C \left[ \dfrac{\alpha}{h^2} \lnorm{\widetilde{\br}_V}{T_V^i}^2 + \lnorm{\widetilde{\br}_V}{T_V^i}^2\right] \\
    & \le \sum_{V \in \cV_H} \sum_{i=1}^8 C \left[ \dfrac{\alpha}{h^2} \lnorm{\br}{T_V^i}^2 + \lnorm{\br}{T_V^i}^2\right] \le C a(\bw, \bw).
    \end{aligned}
\end{equation}
Together with \eqref{eq:rTEstV} and \eqref{eq:rVEstAll}, we have
\begin{equation}\label{eq:rEstAll}
    \sum_{T \in \cT_H} a(\br_T, \br_T) + \sum_{V \in \cV_H} a\left( \br_V , \br_V \right) \le C a(\bw, \bw).
\end{equation}
Next, we consider $\bg = \nabla \, q$. 

Let $\widetilde{\bg}_V$ be defined by 
\begin{equation}
\widetilde{\bg}_V := \nabla \, \left( \nu_V(q) \psi_V + \sum_{i=1}^{12} \sum_{v \in \cV_h^{F_V^i} }\dfrac{1}{\cN_{F_V^i}}\nu_v(q)\psi_v + \sum_{j=1}^6 \sum_{v \in \cV_h^{E_V^j}}\dfrac{1}{\cN_{E_V^j}} \nu_v(q) \psi_v \right).
\end{equation}
Using a standard inverse estimate and a scaling argument, we obtain
\begin{equation}\label{eq:gVinverse}
\lnorm{\widetilde{\bg}_V}{T}^2 \le \dfrac{C}{h^2} \lnorm{q}{T}^2 \qquad \forall T \in \cT_H.
\end{equation}
We then construct $\bg_V \in N_h^V$ so that
\begin{equation}
    \bg_V \cdot \bt_e = \widetilde{\bg}_V \cdot \bt_e \mbox{ for } e \in \left(\bigcup_{i=1}^{12}\cE_h^{F_V^i}\right) \bigcup \left(\bigcup_{j=1}^6 \cE_h^{E_V^j} \right).
\end{equation}
Now that $\bg$ and $\sum_{V \in \cV_H} \bg_V$ have the identical degrees of freedom on the edges in $\bigcup_{T \in \cT_H}\cE_h^{\p T}$, we have
\begin{equation}
    \bg = \sum_{T \in \cT_H} \bg_T + \sum_{V \in \cV_H} \bg_V
\end{equation}
for unique vector fields $\bg_T \in N_h^T$.
For $\bg_T$,  approach with \eqref{eq:gTEstE} to obtain
\begin{equation}\label{eq:gTEstV}
\sum_{T \in \cT_H} a(\bg_T, \bg_T) \le C a(\bw, \bw).
\end{equation}
By the construction of $\bg_V$ and \eqref{eq:MinimumEnergyV}, we obtain
\begin{equation}\label{eq:gVEst1}
    a(\bg_V, \bg_V) \le a(\widetilde{\bg}_V, \widetilde{\bg}_V) = \sum_{i=1}^8 \lnorm{\widetilde{\bg}_V}{T_V^i}^2.
\end{equation}
Moreover, we have the following estimate using \eqref{eq:HD3}, \eqref{eq:gVinverse}, and \eqref{eq:gVEst1}:
\begin{equation}\label{eq:gVEstV}
\sum_{V\in\cV_H} a(\bg_V, \bg_V) \le \dfrac{C}{h^2}\lnorm{q}{\Omega}^2 \le C\lnorm{\bw}{\Omega}^2 \le C a(\bw, \bw).
\end{equation}
From \eqref{eq:gTEstV} and \eqref{eq:gVEstV}, we therefore have
\begin{equation}\label{eq:gEstAll}
    \sum_{T \in \cT_H} a(\bg_T, \bg_T) + \sum_{V \in \cV_H} a\left( \bg_V , \bg_V \right) \le C a(\bw, \bw).
\end{equation}
With $\bw_T = \br_T + \bg_T$ and $\bw_V = \br_V + \bg_V$, we obtain the desired estimate \eqref{eq:StabilityEstimateV} from \eqref{eq:rEstAll} and \eqref{eq:gEstAll}.
\end{proof}

\subsection{Convergence analysis of the V--cycle multigrid algorithms}
We now consider the convergence analysis for the V--cycle multigrid. The error propagation operator $E_k : N_k \longrightarrow N_k$ for the V--cycle multigrid methods with $m$ smoothing steps is given by 
\begin{equation}\label{eq:VCycleRecursion}
    E_k = 
    \begin{cases}
    0 & \mbox{ if } k = 0, \\
    E_k=R_k^m\left(Id_k-I_{k-1}^kP_k^{k-1}\right)R_k^m+R_k^m\left(I_{k-1}^k E_{k-1} P_k^{k-1}\right)R_k^m & \mbox{ if } k\geq 1;
    \end{cases}
\end{equation}
see \cite{Hackbusch:1985:MMA,McCormick:1987:SIAMBook}.
Here, $I_{k-1}^k$ is defined in Section~\ref{subsec:grid} and the operator $P_{k}^{k-1} : N_k \longrightarrow N_{k-1}$ is the Ritz projection operator defined by 
\begin{equation}\label{eq:RitzProjection}
    a\left(P_k^{k-1}\bw,\bv\right)=a\left(\bw,I_{k-1}^k\bv\right)\qquad\forall\,\bw\in N_k,\,\bv\in N_{k-1}.
\end{equation}
Moreover, we define $R_k : N_k \longrightarrow N_k$ by 
\begin{equation}\label{eq:RkDef}
    R_k=Id_k- M_k^{-1}A_k,
\end{equation} 
where $Id_k$ is the identity operator on $N_k$.
\begin{rem}\label{lem:symmetryRkEk}
The operator $R_k$ in \eqref{eq:RkDef} is symmetric with respect to the inner product $a(\cdot, \cdot)$ and $E_k$ is symmetric positive semidefinite with respect to $a(\cdot, \cdot)$. For more detail, see Chapter 6 of \cite{BScott:2008:FEM}.
\end{rem}

We will follow the framework in Bramble and Pasciak \cite{BP:1987:NCE}. We can also refer to Chapter 6 of \cite{BScott:2008:FEM}. We note that the spectral conditions in Section~\ref{subsubsec:edge-based} and Section~\ref{subsubsec:vertex-based} and stability estimates in Lemma~\ref{lem:StabilityE} and Lemma~\ref{lem:StabilityV} play main roles in the framework.

We first consider a smoothing property.
\begin{lem}\label{lem:Smoothing}
For $m \ge 1$, we have
  \begin{equation*}
    a\left(\left(Id_k-R_k\right)R_k^m\bv,R_k^m\bv\right)\leq \frac{1}{2m}a\left(\left(Id_k-R_k^{2m}\right)\bv,\bv\right)
    \qquad\forall\,\bv\in N_k,\;k\geq1.
  \end{equation*}
\end{lem}
\begin{proof}
Let $\bv \in N_k$ be arbitrary. Since $R_k$ is symmetric with respect to the inner product $a(\cdot,\cdot)$,
it follows from the spectral conditions in Section~\ref{subsubsec:edge-based} and Section~\ref{subsubsec:vertex-based} and the spectral theorem that 
\begin{equation*}
    a\left(\left(Id_k-R_k\right)R_k^l\bv,\bv\right)\leq a\left(\left(Id_k-R_k\right)R_k^j \bv,\bv\right)\qquad\text{for}\;0\leq j\leq l,
\end{equation*}
and thus we have
\begin{align*}
(2m)a\left(\left(Id_k-R_k\right)R_k^m\bv,R_k^m\bv\right)&=(2m)a\left(\left(Id_k-R_k\right)R_k^{2m}\bv,\bv\right)\\
&\leq \sum_{j=0}^{2m-1} a\left(\left(Id_k-R_k\right)R_k^j\bv,\bv\right)
=a\left(\left(Id_k-R_k^{2m}\right)\bv,\bv\right).
\end{align*}
\end{proof}
We next derive two approximation properties.

\begin{lem}\label{lem:ApproximationI}
For all $\bv \in N_k$ and $k \ge 1$, let $\bw = \left(Id_k-I_{k-1}^kP_k^{k-1}\right)\bv$.
We then have the following estimates:
\begin{equation*}
    \left\langle M_{E,k} \bw,\bw\right\rangle
    \leq \frac{C_{E,\dag}}{\eta_E} a\left(\bw,\bw\right)
\end{equation*}
and
\begin{equation*}
    \left\langle M_{V,k} \bw, \bw\right\rangle
    \leq \frac{C_{V,\dag}}{\eta_V} a\left(\bw,\bw\right).
\end{equation*}
\end{lem}
\begin{proof}
We will use a well-know additive Schwarz theory. For more details, see Chapter 7 of \cite{BScott:2008:FEM}.
For any $\bw\in N_h$, we have
\begin{equation}\label{eq:FundamentalE}
    \left\langle M_{E,k} \bw,\bw \right\rangle=\eta_E^{-1}\inf_{\substack{\bw=\sum_{T\in\cT_H}\bw_T+\sum_{E\in\cE_H}\bw_E\\
            \bw_T\in N_h^T,\,\bw_E\in N_h^E}}
    \left(\sum_{T\in\cT_H}a\left(\bw_T,\bw_T\right)+\sum_{E\in\cE_H}a\left(\bw_E,\bw_E\right)\right).
\end{equation}
We therefore have the estimate for $M_{E,k}$ from Lemma~\ref{lem:StabilityE} and \eqref{eq:FundamentalE} with $\bw = \left(Id_k-I_{k-1}^kP_k^{k-1}\right)\bv$.

Similarly, for any $\bw \in N_h$, the following relation holds:
\begin{equation}\label{eq:FundamentalV}
    \left\langle M_{V,k} \bw,\bw\right\rangle=\eta_V^{-1}\inf_{\substack{\bw=\sum_{T\in\cT_H}\bw_T+\sum_{V\in\cV_H}\bw_V\\
            \bw_T\in N_h^T,\,\bw_V\in N_h^V}}
    \left(\sum_{T\in\cT_H}a\left(\bw_T,\bw_T\right)+\sum_{V\in\cV_H}a\left(\bw_V,\bw_V\right)\right).
\end{equation}
In a similar way, we obtain the estimate for $M_{V,k}$ from Lemma~\ref{lem:StabilityV} and \eqref{eq:FundamentalV} with $\bw = \left(Id_k-I_{k-1}^kP_k^{k-1}\right)\bv$.
\end{proof}

\begin{lem}\label{lem:ApproximationII}
We have
\begin{equation*}
    a\left(\left(Id_k-I_{k-1}^kP_k^{k-1}\right)\bv,\left(Id_k-I_{k-1}^kP_k^{k-1}\right)\bv\right)
    \leq \frac{C_\dag}{\eta}a\left(\left(Id_k-R_k)\bv,\bv\right)\right)
    \qquad\forall\,\bv\in N_k,\,k\geq 1,
\end{equation*}
where $C_\dag = C_{E,\dag}$ (resp. $C_{V,\dag}$) and $\eta = \eta_E$ 
(resp. $\eta_V$) if $M_k = M_{E,k}$ (resp. $M_{V,k}$).
\end{lem}
\begin{proof}

Let $\bw = \left(Id_k-I_{k-1}^kP_k^{k-1}\right)\bv$. By \eqref{eq:RitzProjection}, Lemma~\ref{lem:ApproximationI} and the Cauchy-Schwarz inequality, we have
\begin{equation*}
\begin{aligned}
a(\bw, \bw) & = a(\bw, \bv) = \left\langle M_{k}\left(M_{k}^{-1}\right)A_k \bv, \bw \right\rangle \\
            & \le \left\langle M_{k}\left(M_{k}^{-1}A_k\right) \bv, \left(M_{k}^{-1}A_k\right) \bv \right\rangle^{1/2}
              \left\langle M_k \bw, \bw\right\rangle^{1/2} \\
            & \le a\left(\left(M_{k}^{-1}A_k\right) \bv, \bv \right)^{1/2} \left(\dfrac{C^\dag}{\eta}\right)^{1/2}a(\bw, \bw)^{1/2} \\
            & = a\left(\left(Id_k - R_k\right) \bv, \bv \right)^{1/2} \left(\dfrac{C^\dag}{\eta}\right)^{1/2}a(\bw, \bw)^{1/2}.
\end{aligned}
\end{equation*}
Hence, we obtain
\begin{equation}
a(\left(Id_k-I_{k-1}^kP_k^{k-1}\right)\bv, \left(Id_k-I_{k-1}^kP_k^{k-1}\right)\bv) \le \dfrac{C^\dag}{\eta} a\left(\left(Id_k - R_k\right) \bv, \bv \right).
\end{equation}
\end{proof}

Finally, we establish our main result, the uniform convergence of the V--cycle multigrid methods.

\begin{thm} \label{thm:Convergence} We have
\begin{equation*}
    \left\|E_k\bw\right\|_a\leq \frac{(C_\dag/\eta)}{(C_\dag/\eta)+2m}\left\|\bw\right\|_a\qquad\forall\,\bw\in N_k,\,k\geq 1,
\end{equation*}
where $C_\dag = C_{E,\dag}$ (resp. $C_{V,\dag}$) and $\eta = \eta_E$ (resp. $\eta_V$) if $M_k = M_{E,k}$ (resp. $M_{V,k}$).
\end{thm}
\begin{proof}
Due to the fact that $E_k$ is symmetric positive semidefinite, it is enough to show that
\begin{equation}\label{eq:EkEst}
    a(E_k\bw,\bw)\leq \frac{C_*}{C_*+2m}a(\bw,\bw)\qquad\forall\,\bw\in V_k,\,k\geq 1,
\end{equation}
where where $C_*=C_\dag/\eta$.

We will prove \eqref{eq:EkEst} by induction. Obviously, the case for $k=0$ holds automatically since $E_0=0$. Let $\delta = C_* / (C_* + 2m)$ and assume that the estimate \eqref{eq:EkEst} is satisfied for $k - 1$. We then have
\begin{align*}
a\left( E_k \bw, \bw\right) &= a\left( R_k^m \left( Id_k -I_{k-1}^k P_k^{k-1} +I_{k-1}^k E_{k-1}P_k^{k-1}\right) R_k^m\bw,\bw\right)\\
    &\le a\left(\left(Id_k -I_{k-1}^k P_k^{k-1}\right) R_k^m \bw, \left(Id_k - I_{k-1}^k P_k^{k-1}\right) R_k^m \bw \right)\\
       & \quad + \delta a\left( P_k^{k-1} R_k^m \bw,  P_k^{k-1} R_k^m \bw\right) \\
    &=\left( 1 - \delta\right) a\left( \left(I d_k- I_{k-1}^kP_k^{k-1} \right) R_k^m \bw, \left(Id_k - I_{k-1}^k P_k^{k-1}\right)R_k^m \bw\right)\\   
     & \quad + \delta a \left(R_k^m \bw,  R_k^m \bw\right)\\
     &\leq \left(1-\delta\right)C_*a\left(\left(Id_k-R_k\right)R_k^m\bw,R_k^m\bw\right)+
     \delta a \left(  R_k^m \bw,  R_k^m \bw\right)\\
      &\leq \left(1-\delta\right)\frac{C_*}{2m}a\left(\left(Id_k-R_k^{2m}\right)\bw,\bw\right)+
     \delta a \left(R_k^m \bw,  R_k^m \bw\right)=\delta a\left(\bw,\bw\right)
\end{align*}
from the induction hypothesis, \eqref{eq:VCycleRecursion}, \eqref{eq:RitzProjection},  Lemma~\ref{lem:Smoothing} and Lemma~\ref{lem:ApproximationII}.
\end{proof}

\section{Numerical experiments}\label{sec:numerics}
In this section, we report the numerical results that support the theoretical estimates and demonstrate the performance of the V--cycle multigrid methods. We use the computational domain $\Omega = (-1, 1)^3$. As the initial triangulation $\cT_0$, we use eight identical unit cubes. 

In the first set of experiments, we carry out the $k-$th level multigrid algorithm with the edge-based smoother introduced in Section~\ref{subsec:smoothers} with $m$ smoothing steps and the damping factor $\eta_E = 1/13$. We compute the contraction numbers for $k = 1, \cdots, 4$ and $m=1, \cdots, 5$. We perform the experiments five times with the coefficient $\alpha = 0.01, 0.1 ,1.0, 10.0, 100.0$. The results are reported in Table~\ref{table:edge_based}. As we see the result, the V--cycle multigrid methods provide uniform convergence. 

In the next set of experiments, we perform similar experiments to the first set of experiments. The only differences are the smoother, the vertex-based smoother, and the damping factor $\eta_V = 1/9$. Other general settings are identical. The contraction numbers are reported in Table~\ref{table:vertex_based}. The results are compatible with our theory and the uniform convergence of the methods is observed.

\begin{table}
    \centering
    \caption{Edge Based Methods}
	{\small\begin{tabular}{c|c|c|c|c|c|c}
		\multicolumn{2}{c|}{} & $m=1$  & $m=2$  & $m=3$  & $m=4$  & $m=5$  \\ \hline
        \multirow{4}{*}{$\alpha = 0.01$}
        &$k=1$ &7.88E-01&6.27E-01&4.44E-01&3.25E-01&3.11E-01\\
        &$k=2$ &8.81E-01&7.79E-01&6.99E-01&5.90E-01&5.62E-01\\
        &$k=3$ &9.24E-01&8.56E-01&7.92E-01&7.36E-01&6.77E-01\\
        &$k=4$ &9.40E-01&8.90E-01&8.41E-01&7.98E-01&7.56E-01\\ \hline
        \multirow{4}{*}{$\alpha = 0.1$}
        &$k=1$ &8.83E-01&7.85E-01&7.03E-01&6.33E-01&5.73E-01\\
        &$k=2$ &9.30E-01&8.70E-01&8.18E-01&7.55E-01&7.25E-01\\
        &$k=3$ &9.53E-01&9.19E-01&8.88E-01&8.52E-01&8.19E-01\\
        &$k=4$ &9.72E-01&9.53E-01&9.35E-01&9.18E-01&9.01E-01\\ \hline
        \multirow{4}{*}{$\alpha = 1.0$}
        &$k=1$ &9.07E-01&8.31E-01&7.69E-01&7.19E-01&6.77E-01\\
        &$k=2$ &9.44E-01&9.17E-01&8.85E-01&8.58E-01&8.30E-01\\
        &$k=3$ &9.70E-01&9.59E-01&9.44E-01&9.30E-01&9.17E-01\\
        &$k=4$ &9.81E-01&9.72E-01&9.65E-01&9.63E-01&9.56E-01\\ \hline
        \multirow{4}{*}{$\alpha = 10.0$}
        &$k=1$ &9.09E-01&8.36E-01&7.77E-01&7.30E-01&6.91E-01\\
        &$k=2$ &9.49E-01&9.25E-01&8.97E-01&8.74E-01&8.55E-01\\
        &$k=3$ &9.72E-01&9.65E-01&9.53E-01&9.42E-01&9.33E-01\\
        &$k=4$ &9.82E-01&9.76E-01&9.73E-01&9.71E-01&9.66E-01\\ \hline
        \multirow{4}{*}{$\alpha = 100.0$}
        &$k=1$ &9.10E-01&8.37E-01&7.78E-01&7.31E-01&6.93E-01\\
        &$k=2$ &9.49E-01&9.26E-01&8.98E-01&8.76E-01&8.57E-01\\
        &$k=3$ &9.73E-01&9.66E-01&9.54E-01&9.43E-01&9.34E-01\\
        &$k=4$ &9.82E-01&9.76E-01&9.73E-01&9.72E-01&9.67E-01\\
        \hline
	\end{tabular}}
    \label{table:edge_based}
\end{table}

\begin{table}
    \centering
    \caption{Vertex Based Methods}
	{\small\begin{tabular}{c|c|c|c|c|c|c}
		\multicolumn{2}{c|}{} & $m=1$  & $m=2$  & $m=3$  & $m=4$  & $m=5$  \\ \hline
        \multirow{4}{*}{$\alpha = 0.01$}
        &$k=1$ &7.90E-01&6.24E-01&4.93E-01&3.90E-01&3.08E-01\\ 
        &$k=2$ &7.91E-01&6.26E-01&4.94E-01&3.92E-01&3.12E-01\\ 
        &$k=3$ &7.90E-01&6.24E-01&4.93E-01&3.90E-01&3.08E-01\\ 
        &$k=4$ &7.90E-01&6.25E-01&4.94E-01&3.91E-01&3.09E-01\\  \hline
        \multirow{4}{*}{$\alpha = 0.1$}
        &$k=1$ &7.90E-01&6.24E-01&4.93E-01&3.90E-01&3.08E-01\\
        &$k=2$ &7.91E-01&6.25E-01&4.94E-01&3.91E-01&3.10E-01\\ 
        &$k=3$ &7.91E-01&6.26E-01&4.95E-01&3.91E-01&3.10E-01\\ 
        &$k=4$ &7.91E-01&6.26E-01&4.95E-01&3.92E-01&3.11E-01\\  \hline
        \multirow{4}{*}{$\alpha = 1.0$}
        &$k=1$ &7.90E-01&6.24E-01&4.93E-01&3.90E-01&3.08E-01\\ 
        &$k=2$ &7.91E-01&6.26E-01&4.95E-01&3.92E-01&3.10E-01\\ 
        &$k=3$ &7.91E-01&6.26E-01&4.95E-01&3.92E-01&3.11E-01\\ 
        &$k=4$ &7.91E-01&6.26E-01&4.95E-01&3.92E-01&3.11E-01\\  \hline
        \multirow{4}{*}{$\alpha = 10.0$}
        &$k=1$ &7.90E-01&6.24E-01&4.93E-01&3.90E-01&3.08E-01\\ 
        &$k=2$ &7.91E-01&6.26E-01&4.95E-01&3.92E-01&3.10E-01\\ 
        &$k=3$ &7.91E-01&6.26E-01&4.95E-01&3.92E-01&3.11E-01\\ 
        &$k=4$ &7.91E-01&6.26E-01&4.95E-01&3.92E-01&3.11E-01\\  \hline
        \multirow{4}{*}{$\alpha = 100.0$}
        &$k=1$ &7.90E-01&6.24E-01&4.93E-01&3.90E-01&3.08E-01\\ 
        &$k=2$ &7.91E-01&6.26E-01&4.95E-01&3.92E-01&3.10E-01\\ 
        &$k=3$ &7.91E-01&6.26E-01&4.95E-01&3.92E-01&3.11E-01\\ 
        &$k=4$ &7.91E-01&6.26E-01&4.95E-01&3.92E-01&3.11E-01\\ 
        \hline
	\end{tabular}}
    \label{table:vertex_based}
\end{table}

We note that a part of implementations is based on the MFEM library; see \cite{mfem, mfemweb} for more details. The implemented codes are available at \url{https://github.com/duksoon-open/MG_ND}.

\section{Concluding remarks}\label{sec:conclusions}
In this work, new multigrid methods based on nonoverlapping domain decomposition smoothers for vector field problems posed in $\HCurl$ have been developed and analyzed. The suggested methods provide uniform convergence and the numerical experiments are consistent with the theoretical results.

There are a few challenges. In our convergence analysis, we assumed that the coefficients are constants and the domain is convex. The numerical results in \cite{Oh:MGHcurlNE} show that the V--cycle multigrid methods work well without the assumptions, i.e. constant coefficients and convex domain. Our theory can therefore be extended to coefficients with jumps or nonconvex domains. We believe that the results in \cite{HU:2021:HXJC, HP:2020:regulardecomp} would be good ingredients for establishing the stronger convergence analysis. 

\section*{Acknowledgement}
This work was supported by research fund of Chungnam National University.

\end{document}